\documentclass[12pt]{amsart}
\usepackage{amssymb,graphics}
\usepackage{hyperref}
\usepackage{tabularx}
\usepackage{booktabs}
\usepackage{caption}
\usepackage{mathdots}
\usepackage[usenames,dvipsnames]{xcolor}

\newtheorem{thm}{Theorem}[section]
\newtheorem{lem}[thm]{Lemma}
\newtheorem{cor}[thm]{Corollary}
\newtheorem{prop}[thm]{Proposition}
\newtheorem{ex}[thm]{Example}

\newtheorem*{prob*}{Open problem}

\theoremstyle{definition}

\newtheorem{defi}[thm]{Definition}

\theoremstyle{remark}

\newtheorem{rem}[thm]{Remark}
\newtheorem*{rem*}{Remark}

\DeclareMathOperator{\Hom}{Hom}

\newcommand{\kringel}{\mathbin{\raise1pt\hbox{$\scriptstyle\circ$}}}
\newcommand{\pkt}{\mathbin{\raise0pt\hbox{$\scriptstyle\bullet$}}}

\newcommand{\Der}{{\rm Der}}

\newcommand{\La}{\mathfrak{a}}
\newcommand{\Lb}{\mathfrak{b}}

\newcommand{\Lg}{\mathfrak{g}}
\newcommand{\Lh}{\mathfrak{h}}

\newcommand{\abs}[1]{\lvert#1\rvert}

\newcommand{\al}{\alpha}
\newcommand{\be}{\beta}

\newcommand{\ra}{\rightarrow}

\renewcommand{\phi}{\varphi}

\begin{document}


\title[Index]{The index of nilpotent Lie algebras}
\author[D. Burde]{Dietrich Burde}
\author[K. Dekimpe]{Karel Dekimpe}
\address{Fakult\"at f\"ur Mathematik\\
Universit\"at Wien\\
Oskar-Morgenstern-Platz 1\\
1090 Wien \\
Austria}
\email{dietrich.burde@univie.ac.at}
\address{Katholieke Universiteit Leuven Campus Kulak Kortrijk\\
Etienne Sabbelaan 53\\
8500 Kortrijk\\
Belgium}
\email{karel.dekimpe@kuleuven.be}
\date{\today}


\begin{abstract}
The index of a Lie algebra is an important invariant which arises in several areas, e.g. in the 
study of coadjoint orbits for a Lie group, in invariant theory and in representation theory. 
We study the index for several classes of nilpotent Lie algebras.
In particular, we give explicit formulas for free-nilpotent Lie algebras of nilpotency class 
two and three, or of solvability class two, for graph Lie algebras, and for filiform nilpotent 
Lie algebras.
\end{abstract}

\maketitle

\section{Introduction}

The index of a Lie algebra was introduced by Dixmier \cite{DIX} because of its importance in representation theory.
For a connected Lie group $G$ with associated Lie algebra $\Lg$ one considers the stabilizer $G(\ell)$ of the coadjoint action,
for $\ell\in \Lg^*$. Its Lie algebra is denoted by $\Lg(\ell)$, and the {\em index of $\Lg$} is defined by
\[
\chi(\Lg)=\inf \{ \dim \Lg(\ell)\mid \ell\in \Lg^*\}.
\]
The stabilizer arises in the study of irreducible unitary representations of nilpotent Lie groups, in Kirillov theory,
in invariant theory, and in several other areas. There are also connections to harmonic analysis, in the context of square-integrable
irreducible unitary representations of connected and simply connected nilpotent Lie groups. 
For references see \cite{BU79,COG}. \\[0.2cm]
However, the general definition of $\chi(\Lg)$ is given over an arbitrary field $K$ without referring to the coadjoint action of a Lie group,
by writing $\Lg(\ell)$ as the radical of the skew-symmetric bilinear form $B_{\ell}(x,y)=\ell([x,y])$. It can be generalized to
the index of a representation of a Lie algebra.  \\[0.2cm]
The index of Lie algebras has been studied by several authors in various contexts, see for example \cite{ADM,CHM,DEK,OOM2,PAN,TAY}
and the references cited therein. Despite of the results obtained, it still remains very difficult to
compute the index of an arbitrary Lie algebra. If $\Lg$ is reductive, then the index coincides with the rank \cite{TAY}. But for
solvable and nilpotent Lie algebras there are only a few explicit results known for the index. Furthermore, the index of a
Lie algebra is also related to another invariant of $\Lg$, namely to 
$\al(\Lg)$, the maximal dimension of an abelian subalgebra of $\Lg$.
This invariant has also been studied in the literature, see \cite{BU39}, and the 
references given therein. \\[0.2cm]
In this paper we prove explicit formulas for the index of free-nilpotent Lie algebras $F_{g,c}$ of nilpotency class $c\ge 2$
with $g$ generators, where $c=2$ and $c=3$, for the index of two-step nilpotent graph Lie algebras, for the index of metabelian free-nilpotent
Lie algebras $M_{g,c}$, and for filiform nilpotent Lie algebras.

\section{Preliminaries}

Let $\Lg$ be a finite-dimensional Lie algebra over a field $K$. We will always assume that $K$ has characteristic zero, if not
stated otherwise. Denote by $Z(\Lg)$ the center of $\Lg$,
and by $\Lg^*$ the vector space $\Hom(\Lg,K)$. Each $\ell\in \Lg^*$ determines a skew-symmetric bilinear form
\[
B_{\ell}\colon \Lg\times \Lg \ra K, \; (x,y)\mapsto \ell([x,y]).
\]  
The radical of $B_{\ell}$ is given by the subalgebra
\[
\Lg(\ell)=\{y\in \Lg\mid \ell ([x,y])=0 \; \forall\, x\in \Lg\}
\]  
of $\Lg$. Obviously $\Lg(\ell)$ contains the ideal $Z(\Lg)$ of $\Lg$, but need not be an ideal itself in general, see
\cite{COG}, Example $1.3.10$. Furthermore, $\Lg(\ell)$ has even codimension in $\Lg$, see Lemma $1.3.2$ in \cite{COG}.

\begin{defi}
The {\em index} of a Lie algebra $\Lg$ over a field $K$ is defined by
\[
\chi(\Lg)=\inf \{ \dim \Lg(\ell) \mid \ell\in \Lg^*\}.
\]  
\end{defi}  

The definition goes back to Dixmier, see \cite{DIX}, 11.1.6. There is a large literature on the index.
In the book  \cite{TAY}, by Tauvel and Yu, several results are listed and references are given.
For reductive Lie algebras, the index equals the rank. For an arbitrary solvable or nilpotent Lie algebra computing the index
is a difficult problem. The index is also related to the invariants $\alpha(\Lg)$ and $\be(\Lg)$ of a Lie algebra $\Lg$.
They are defined as follows.

\begin{defi}
Let $\Lg$ be a Lie algebra over a field $K$. The maximal dimension of an {\em abelian subalgebra} of $\Lg$ is
denoted by $\alpha(\Lg)$. The maximal dimension of an {\em abelian ideal} of $\Lg$ is denoted by $\be(\Lg)$. 
\end{defi}  

Let us summarize a few properties of the invariants $\al(\Lg)$ and $\chi(\Lg)$.

\begin{prop}\label{2.3}
Let $\Lg$ be a Lie algebra over an algebraically closed field $K$ of characteristic zero. Then 
$\al(\Lg)$ and $\chi(\Lg)$ satisfy the following properties.
\begin{itemize}
\item[(1)] $\dim Z(\Lg) \le \chi(\Lg)\le \dim \Lg$.
\item[(2)] $\al(\Lg)=\be(\Lg)$ for every solvable Lie algebra over $K$.
\item[(3)] $\chi(\Lg)\le \al(\Lg)$.
\item[(4)] $\chi(\Lg)\ge 2\al(\Lg)-\dim \Lg$.
\item[(5)] If $\Lg$ degenerates to $\Lh$, then $\al(\Lg)\le \al(\Lh)$ and $\chi(\Lg)\le \chi(\Lh)$.
\item[(6)] Let $\{x_1,\ldots ,x_n\}$ be a basis of $\Lg$. Then the index of $\Lg$ is given by
\[
n - {\rm rank}_R M(\Lg),
\]
where $M(\Lg)=([x_i,x_j])\in M_n(R)$, and $R$ is the quotient field of $K[x_1,\ldots ,x_n]$.
\end{itemize}
\end{prop}

\begin{proof}
Clearly we have $\chi(\Lg)\le \dim \Lg$ because $\Lg(\ell)$ is a subalgebra of $\Lg$. 
Since $Z(\Lg)\subseteq \Lg(\ell)$ for all $\ell\in \Lg^*$, property $(1)$ is satisfied. Note that it holds over an
arbitrary field. For property $(2)$ see Proposition $2.6$ in \cite{BU39}. There are counterexamples for fields, which are
not algebraically closed. Properties $(3)$ and $(4)$ are stated in \cite{OOM2}, formula $(3)$ in section $1.2$ with a
reference to \cite{OOM1}, Theorem $14$ for a proof. 
For property $(5)$ for $\al(\Lg)$ see \cite{NEP}, Theorem $1$, part $2$. The index is an upper semi-continuous function,
see \cite{PAN}, section $1$. Hence we have $\chi(\Lg)\le \chi(\Lh)$ for $\Lg\ra_{deg}\Lh$.
For property $(6)$ also see  \cite{PAN}, section $1$.
\end{proof}  

For a Lie algebra $\Lg$, denote by $\Lg^1=\Lg$ and $\Lg^{k+1}=[\Lg,\Lg^k]$ for $k\ge 1$ the terms of the lower central series.

\begin{defi}
Let $F_g$ denote the free Lie algebra on $g$ generators. Then 
\[
F_{g,c}:=F_g/F_g^{c+1}
\]
is called the {\em free-nilpotent Lie algebra} with $g\ge 1$ generators and nilpotency class $c\ge 1$.
\end{defi}  

If $g=1$, then $F_{g,c}$ is $1$-dimensional and abelian. Hence we will assume that $g\ge 2$.
Also, for $c=1$, the Lie algebra $F_{g,c}$ is abelian of dimension $g$. So let $c\ge 2$.
The dimension of $F_{g,c}$ is given by the {\em Witt formula}.

\begin{prop}
The dimension of  $F_{g,c}$ and its center are given by
\begin{align*}
\dim (F_{g,c}) & =\sum_{m=1}^c \frac{1}{m}\sum_{d\mid m}\mu(d)g^{\frac{m}{d}},\\[0.1cm]
\dim Z(F_{g,c})& =\frac{1}{c}\sum_{d\mid c} \mu(d)g^{\frac{c}{d}}.
\end{align*}

\end{prop}

\section{The index of free-nilpotent Lie algebras}

We first consider the free-nilpotent Lie algebras $F_{g,2}$ of nilpotency class $c=2$, with
$g\ge 2$ generators. We have $\dim F_{g,2}=\binom{g}{2}+g=\binom{g+1}{2}$, and $Z(F_{g,2})=[F_{g,2},F_{g,2}]$ is of
dimension $\binom{g}{2}$.

\begin{prop}
We have $\al(F_{g,2})=\binom{g}{2}+1$ for all $g\ge 2$.
\end{prop}  

\begin{proof}
Let $\La$ be an abelian subalgebra of maximal dimension in $F_{g,2}$. Since $\La$ is self-centralizing, we have
$Z(F_{g,2})\subseteq Z_{F_{g,2}}(\La)=\La$, where $Z_{\Lg}(\La)=\{x\in \Lg\mid [x,\La]=0\}$ is the centralizer of $\La$ in $\Lg$.
For any element $v_1\in F_{g,2}\backslash Z(F_{g,2})$ the linear span  $\langle v_1,Z(F_{g,2})\rangle$ is an abelian
subalgebra of  $F_{g,2}$. However, if we now take any $v_2 \in F_{g,2}\backslash \langle v_1, Z(F_{g,2}) \rangle$, then $[v_1,v_2]\neq 0$. 
Indeed, suppose that $[v_1,v_2]=0$. We can pick $v_3, v_4, \ldots, v_g \in F_{g,2}$ in such a 
way that the elements 
\[
v_1 + Z(F_{g,2}), \;  v_2+ Z(F_{g,2}),  \ldots, v_g  + Z(F_{g,2}) 
\]
form a basis of $F_{g,2}/Z(F_{g,2})$. It follows that 
$[F_{g,2}, F_{g,2}]$ is spanned by the elements $[v_i,v_j]$ for  $1\le i<j\le g$. But as $[v_1,v_2]=0$, this would imply that 
the dimension of $[F_{g,2}, F_{g,2}]$ is at most $\binom{g}{2}-1$ which is a contradiction. Hence $[v_1,v_2]\neq 0$ which shows that 
the maximal dimension of an abelian subalgebra of $F_{g,2}$ is $\dim Z(F_{g,2})+1$.
\end{proof}  

It is also possible to give an explicit formula for the index of $F_{g,2}$.

\begin{prop}\label{3.2}
The index of the two-step nilpotent Lie algebra $F_{g,2}$ with $g\ge 2$ is given by
\[
\chi (F_{g,2})=\begin{cases} \binom{g}{2}, & \text{ if }  g\equiv 0 \bmod 2 \\[0.1cm]
\binom{g}{2}+1, & \text{ if } g\equiv 1 \bmod 2 \end{cases}
\]
\end{prop}

\begin{proof}
Let $\Lg=F_{g,2}$. The rank of the matrix $M(\Lg)$ in $(6)$ of Proposition $\ref{2.3}$ is the rank of a general skew-symmetric matrix of
size $g$. Indeed, Consider a basis of $\Lg$ consisting of the generators $x_1,\ldots ,x_g$ supplemented with a basis of the center. Then the 
last $\binom{g}{2}$ rows and columns of $M(\Lg)$ are $0$ while for the $g \times g$ block in the upper left corner, the  entries of the upper-diagonal part
are exactly  the $\binom{g}{2}$ Lie brackets $[x_i,x_j]$ for $1\le i<j\le g$ 
(and we have $-[x_i,x_j]$ in the lower-diagonal part).
It is well known that such a matrix has full rank $g$, if $g$ is even, and rank $g-1$ if $g$ is odd. Indeed, the rank cannot be smaller
than the rank for any specialization.
So consider the skew-symmetric matrix with entries $1$ instead of $[x_i,x_j]$ for $i<j$, which has rank $g$ for even $g$, and
rank $g-1$ for odd $g$. By Proposition $\ref{2.3}$, part $(6)$ the index is given by
\[
\chi (F_{g,2})=\begin{cases} \binom{g}{2}+ g - g = \binom{g}{2}  & \text{ if }  g\equiv 0 \bmod 2 \\[0.1cm]
\binom{g}{2}+ g - (g-1) = \binom{g}{2}+1 & \text{ if } g\equiv 1 \bmod 2 \end{cases}
\]
\end{proof}

For example, the Heisenberg Lie algebra $F_{2,2}$ has index $1$. The results show that the above inequalities
$(3)$ and $(4)$ can be strict. In fact, for $g$ even we have $\chi(F_{g,2})< \al(F_{g,2})$, and for $g\ge 3$ odd we have
$\chi(F_{g,2})> 2\al(F_{g,2})-\dim F_{g,2}$, but $\chi(F_{g,2})=\al(F_{g,2})$. \\[0.2cm]
Next we want to compute the index of the $3$-step nilpotent algebras $F_{g,3}$. We have
\begin{align*}
\dim F_{g,3} & = \frac{g(g+1)(2g+1)}{6},\\[0.1cm] 
\dim Z(F_{g,3}) & = \frac{g^3-g}{3}.
\end{align*}

With respect to the index the case $g=2$ is different from the case $g\ge 3$. Let $x_1,x_2$ be the generators of $F_{2,3}$, and
$x_3=[x_1,x_2], \, x_4=[x_1,[x_1,x_2]]=[x_1,x_3], \, x_5=[x_2,[x_1,x_2]]=[x_2,x_3]$. 

\begin{ex}\label{3.3}
Let $\Lg=F_{2,3}$. Then the matrix
\[
M(\Lg)=\begin{pmatrix} 0 & x_3 & x_4 & 0 & 0 \\ -x_3 & 0 & x_5 & 0 & 0 \\ -x_4 & -x_5 & 0 & 0 & 0 \\ 0 & 0 & 0 & 0 & 0 \\
0 & 0 & 0 & 0 & 0 \\ \end{pmatrix}
\]
has rank $2$, so that $\chi(\Lg)=3$. Moreover we have $\al(\Lg)=3$.
\end{ex}

Indeed, this follows from  Proposition $\ref{2.3}$, part $(6)$. Furthermore, $\La=\langle x_3,x_4,x_5\rangle $ is an abelian
subalgebra of maximal dimension. \\[0.2cm]
The general result for the index of  $F_{g,3}$ with $g\ge 3$ is as follows.

\begin{thm}\label{3.4}
The index of the free-nilpotent Lie algebra $F_{g,3}$ for $g\ge 3$ is given by
\begin{align*}
\chi (F_{g,3}) & = \dim Z(F_{g,3}) + \binom{g}{2}-g \\[0.1cm]
               & = \dim F_{3,g}-2g \\[0.1cm] 
               & = \frac{2g^3+3g^2-11g}{6}.   
\end{align*}
\end{thm}  

\begin{proof}
Let $x_1,x_2,\ldots ,x_g$ be the generators of $F_{g,3}$, and let
\begin{align*}  
x_{ij} & = [x_i,x_j], \hspace*{0.95cm} 1\le i<j\le g, \\
x_{ijk} & = [x_i,[x_j,x_k]], \; 1\le j<k\le g,\, 1\le i\le k\le g.
\end{align*}
Then these elements $x_i,x_{ij},x_{ijk}$ form a basis of $F_{g,3}$. Note that if $k<i$ and $j<k$, then
\begin{align*} 
[x_i,[x_j,x_k]] 
       & = -[x_j,[x_k,x_i]]+[x_k,[x_j,x_i]] \\
       & = -x_{jki}+x_{kji}
\end{align*}  
in terms of the basis elements.
Now let $\ell\in \Lg^*$ and write $x$ in $F_{g,3}$ as
\[
x=\sum_{1\le k \le g}a_kx_k+\sum_{1\le k<\ell\le g}b_{k\ell}x_{k\ell}+z,
\]
where $z\in Z(F_{g,3})=\langle x_{ijk} \mid 1\le j<k\le g,\, 1\le i\le k\le g\rangle$. We want to estimate the dimension
of the subalgebra $\Lg(\ell)$. We have $x\in \Lg(\ell)$ if and only if
\begin{align*}
\ell([x_i,x]) & = 0 \text{ for all } 1\le i \le g, \\
\ell([[x_i,x_j],x]) & = 0 \text{ for all } 1\le i<j\le g.
\end{align*} 
Note that the central element $z$ here plays no role, and can be chosen arbitrarily. These conditions are equivalent to a
system of equations 
\begin{align*}
0 & = \sum_{1\le k\le g} a_k \ell([x_i,x_k]) + \sum_{1\le k<\ell\le g} b_{k\ell} \ell([x_i,x_{kl}), \quad  1\le i \le g\\
0 & = \sum_{1\le k\le g}a_k \ell([[x_i,x_j],x_k]), \quad 1\le i<j\le g
\end{align*}
in the $g$ variables $a_k$ and $\binom{g}{2}$ variables $b_{k\ell}$. The variables $b_{k\ell}$ only appear in the first set
of equations, which are $g$ equations. So there are at least  $\binom{g}{2}-g$ degrees of freedom when determining
the  variables $b_{k\ell}$. Together with the fact that $z\in Z(F_{g,3})$ can be chosen freely, this shows that
\[
\dim \Lg(\ell)\ge \dim Z(F_{3,g})+ \binom{g}{2}-g
\]
for all  $\ell\in \Lg^*$. It remains to show that there exists some $\ell\in \Lg^*$, for which we have equality above.
Denote by $x_{ijk}^*$ the dual linear form in $\Lg^*$ corresponding to the basis vector $x_{ijk}$.
Then define an $\ell\in \Lg^*$ as follows:
\[
\ell=\sum_{\substack{1\le j<k\le g \\ 1\le i\le k\le g}}(i+j+k)x_{ijk}^*.
\]  
For $i>k$ and $1\le j<k\le g$ we have
\begin{align*}
\ell([x_i,[x_j,x_k]]) & = \ell(-x_{jki})+\ell(x_{kji}) \\
                      & = -(j+k+i)+(k+j+i) \\
                      & = 0.
\end{align*}  
As shown above, $x\in \Lg(\ell)$ is equivalent to two sets of linear equations. For example, we have
$\ell([x_1,x_2],x)=0$, which is equivalent to
\begin{align*}
0 & = a_1\ell([[x_1,x_2],x_1])+a_2\ell([[x_1,x_2],x_2])+\ldots + a_g\ell([[x_1,x_2],x_g]) \\
\Leftrightarrow 0 & = a_1\ell([x_1,[x_1,x_2]])+a_2\ell([x_2,[x_1,x_2]]) \\
\Leftrightarrow 0 & = 4a_1+5a_2,
\end{align*}  
where we have used that $\ell([[x_1,x_2],x_i])=0$ for all $i>2$, see above. Similarly, $\ell([x_2,x_3],x)=0$ and
$\ell([x_1,x_3],x)=0$ each yield a linear equation in the $a_k$, so that we have the system
\begin{align*}
0 & = 4a_1+5a_2, \\
0 & = 5a_1+6a_2+7a_3,\\
0 & = 6a_1+7a_2+8a_3.
\end{align*}  
The determinant of the corresponding matrix is equal to $6$, so that we have $a_1=a_2=a_3=0$. By induction it follows
that $a_i=0$ for all $1\le i\le g$. For this, assume that $a_1=\cdots =a_{k-1}=0$ for some $4\le k \le g$. Then we must have
$\ell([[x_1,x_k],x])=0$, which means
\[
(k+2)a_1+(k+3)a_3+\cdots +\cdots +(2k)a_{k-1}+ (2k+1)a_k=0.
\]  
Hence we have $a_k=0$. So we are left with the system of equations
\[
0 = \sum_{1\le k<\ell\le g} b_{k\ell} \ell([x_i,x_{kl}]), \quad  1\le i \le g.
\]
Let us order the variables as follows. We start with $b_{12}$, then $b_{13},b_{23}$, then $b_{14},b_{24},b_{34}$ and so on.
The first three equations then are given by
\begin{align*}
0 & = 4b_{12}+5b_{13}+6b_{23}+\cdots , \\
0 & = 5b_{12}+6b_{13}+7b_{23}+\cdots , \\
0 & = 7b_{13}+8b_{23}+\cdots . \\
\end{align*}  
The other equations do not involve the variables $b_{12}, b_{13},b_{23}$ anymore. All remaining equations are of the form
$(2k+1)b_{1k}+ \cdots $ for $k\ge 4$, so that the matrix corresponding to all $g$ equations 
has block form. The first block is the above $3\times 3$-block, and the second one is the $(g-3)\times (g-3)$-block, which is
an upper-triangular matrix, with nonzero entries on the diagonal. Since the block of size $3$ has again 
determinant $6$, so rank $3$, it is clear that the matrix has rank $g$. Hence the system of the equations 
in the variables $b_{k\ell}$ has a space of solutions of dimension $\binom{g}{2}-g$, so that
\[
\dim \Lg(\ell)= \dim Z(F_{3,g})+ \binom{g}{2}-g
\]  
for the $\ell$ defined above. This finishes the proof.
\end{proof}

\begin{cor}
Let $K$ be an algebraically closed field of characteristic zero and $g\ge 3$. Then we have
\begin{align*}
\al(F_{g,3}) & = \dim Z(F_{g,3})+\binom{g}{2} \\[0.1cm]
            & = \chi(F_{g,3})+g, \\[0.1cm] 
             & = \frac{2g^3+3g^2-5g}{6}.
\end{align*}  
\end{cor}  

\begin{proof}
Let $\Lg=F_{g,3}$. Then $[\Lg,\Lg]$ is an abelian subalgebra of $\Lg$ of dimension $\dim Z(\Lg)+\binom{g}{2}$.
This implies
\[
\al(\Lg)\ge \dim Z(\Lg)+\binom{g}{2}.
\]  
On the other hand, $(4)$ of Proposition $\ref{2.3}$ implies that
\begin{align*}
\al(\Lg) & \le \frac{\chi(\Lg)+\dim \Lg}{2} \\[0.1cm]
            & = \frac{2g^3+3g^2-5g}{6} \\[0.1cm]
            & = \dim Z(\Lg)+\binom{g}{2}.
\end{align*}  
\end{proof}  

\begin{rem}
For  nilpotency class $c\ge 4$ it seems to be difficult to obtain similar formulas for the index of $F_{g,c}$.
We have computed the index for a few examples with $c=4$ and $c=5$, using $(6)$, i.e., by computing the nullspace and
the rank of $M(\Lg)$:
\vspace*{0.3cm}
\begin{center}
\begin{tabular}{c|c|c|c|c}
$\Lg$  & $F_{2,4}$ & $F_{3,4}$ & $F_{4,4}$ & $F_{3,5}$  \\
\hline
$\dim \Lg$      & $8$  & $32$ & $90$ & $80$  \\
$\dim Z(\Lg)$   & $3$  & $18$ & $60$ & $48$   \\
${\rm rank}\; M(\Lg)$  & $4$  & $8$ & $14$ & $12$   \\
$\chi(\Lg)$    & $4$  & $24$ & $76$ & $68$ \\
\end{tabular}
\end{center}
\end{rem}  

However, we are able to prove a general formula for the index of the metabelian quotients of the
Lie algebras $F_{g,c}$, see section $5$.

\section{The index of graph Lie algebras}

Let $K$ be a field, and $G(V,E)$ be a finite simple graph with vertex set $V$ and edge set $E$.
Associate to $G(V,E)$ a $2$-step nilpotent Lie algebra $\Lg$ over $K$ as follows. The underlying vector space for $\Lg$
is $U\oplus W$, where $U$ is the $K$-vector space with basis $V$, and $W$ is the subspace of $\Lambda^2(U)$ which is spanned
by the elements
\[
\{ v_1\wedge v_2 \mid \{v_1,v_2\}\in E\}.
\]  
Then $\dim \Lg=\abs{V}+\abs{E}$, and the Lie brackets for $\Lg$ are given by
\[
[v_1,v_2] =\begin{cases} v_1\wedge v_2, & \text{ if }  \{v_1,v_2\}\in E,   \\[0.1cm]
0, & \text{ if }  \{v_1,v_2\}\not\in E.  \end{cases}
\]
for $v_1,v_2\in V$, where we require that $W\subseteq Z(\Lg)$. Note that we have $[\Lg,[\Lg,\Lg]]=0$.

\begin{ex}
Let $G(V,E)$ be the complete graph on $V$, with $g=\abs{V}\ge 2$, i.e., such that every pair of distinct vertices
is connected by a unique edge. Then the associated Lie algebra equals $F_{g,2}$.
\end{ex}  

\begin{defi}
A {\em matching} of a simple graph $G(V,E)$ is a subset $S\subseteq E$ such that each vertex of $V$ belongs to at most one
element of $S$. A {\em maximum match} of $G(V,E)$ is a matching with a maximum number of edges.
\end{defi}  

\begin{defi}
The {\em matching number} of a simple graph $G(V,E)$ is the number of edges in a maximum match of
$G(V,E)$. It is denoted by $\nu(G(V,E))$.
\end{defi} 

We have the following result on the index of a graph Lie algebra.

\begin{prop}\label{4.4}
Let $\Lg$ be the $2$-step nilpotent Lie algebra over a field $K$ associated to a finite simple graph $G(V,E)$. Then
the index of $\Lg$ is given by
\begin{align*}
\chi(\Lg) & = \dim \Lg-2\nu(G(V,E)) \\[0.1cm]
          & = \abs{V}+\abs{E}-2\nu(G(V,E)).
\end{align*}
\end{prop}  

\begin{proof}
Let $V=\{x_1,x_2,\ldots ,x_n\}$. This is a basis for $U$. Then $W$ is spanned by all elements $[x_i,x_j]=x_i\wedge x_j$ with
$\{x_i,x_j\}\in E$. As a basis for $W$ choose all elements
\[
x_{ij}=[x_i,x_j]=x_i\wedge x_j, \, i<j,
\]
where $\{x_i,x_j \}\in E$. For $i>j$ and $\{x_i,x_j \}\in E$ we have $[x_i,x_j]=-x_{ij}$. So the elements $x_i$ for $1\le i\le n$
and the $x_{ij}$ for $i<j$ together form a basis of $\Lg$. Choose an ordering of this basis as follows:
\[
x_1,x_2,\ldots x_n,x_{n+1},x_{n+2},\ldots ,x_{n+m}
\]
where $m=\abs{E}$, $n=\abs{V}$, and each $x_{n+k}$ is one of the $x_{ij}$. We have
\[
\chi(\Lg)=\dim \Lg - {\rm rank}_R M(\Lg)
\]
by $(6)$ of Proposition $\ref{2.3}$, with the matrix $M(\Lg)=([x_i,x_j])\in M_{n+m}(R)$. This matrix has the last $m$ rows and columns
equal to zero, because $x_{n+k}\in Z(\Lg)$ for all $k$. So we may delete these rows and columns and pass to the matrix $M=(M_{ij})$, with
\[
M_{ij} =\begin{cases} x_{ij}, & \text{ if }  \{x_i,x_j\}\in E, \; i<j,  \\[0.1cm]
-x_{ij} &  \text{ if }  \{x_i,x_j\}\in E, \; i>j,  \\[0.1cm]
  0, & \text{ if }  \{x_i,x_j\}\not\in E.  \end{cases}
\]
This matrix is well known in graph theory and is called the {\em skew-symmetric adjacency matrix} of $G(V,E)$.
It is a result of Lov\'asz \cite{LOV}, that ${\rm rank}(M)=2\nu (G(V,E))$. This finishes the proof.
\end{proof}

\begin{rem}
Let $\{\{i_1,j_1\}\,\; \{i_2,j_2\},\; \ldots, \; \{i_m,j_m\} \}$ be a match in $G(V,E)$. 
Using the notation from the proof above (and assuming without loss of generality that $i_k<j_k$ for all $k$) we let $\ell = x_{i_1j_1}^\ast + x_{i_2j_2}^\ast + \cdots + x_{i_mj_m}^\ast$, where 
$x_{i_kj_k}^\ast$ is the dual linear map corresponding to the basis vector $x_{i_kj_k}$.
Then it is easy to see that 
$\dim \Lg(\ell)=\dim \Lg - 2 m $. So if the matching is maximal, this gives us a concrete $\ell$ with $\dim \Lg(\ell) = \chi(\Lg)$.
\end{rem}

\begin{rem}
The proposition above recovers Proposition $\ref{3.2}$. In fact, for a complete graph $G(V,E)$ with $\abs{V}=g$ and
$\abs{E}=\binom{g}{2}$, the index of $\Lg=F_{g,2}$ given by Proposition $\ref{4.4}$ is
$\chi(\Lg)=\dim \Lg-2\nu (G(V,E))$, where $2\nu (G(V,E))={\rm rank} M(\Lg)$. 
\end{rem}

\section{The index of metabelian free-nilpotent Lie algebras}

A Lie algebra $\Lg$ is called {\em metabelian}, if $[[\Lg,\Lg],[\Lg,\Lg]]=0$. It is also called $2$-step solvable in this case.
Denote by $M_{g,c}$ the free $c$-step nilpotent and metabelian Lie algebra with $g$ generators. We have
\[
M_{g,c}=F_{g,c}/[F_{g,c}^2,F_{g,c}^2],
\]  
where $[F_{g,c}^2,F_{g,c}^2]\subseteq F_{g,2}^4$. In particular, all algebras $F_{g,2}$ and  $F_{g,3}$ are metabelian.
For computing the index of $M_{g,c}$ we will use the following result of Ooms, see \cite{OOM1}, Theorem $14$, part $2$.

\begin{prop}\label{5.1}
Let $\Lg$ be a Lie algebra with basis $x_1,\ldots ,x_n$, and let $\Lh$ be an abelian subalgebra with basis
$h_1,\ldots ,h_m$. Then the following conditions are equivalent.
\begin{itemize}
\item[(1)] $\chi(\Lg)=2\dim \Lh-\dim \Lg$. 
\item[(2)] ${\rm rank}_R([x_i,h_j])=\dim \Lg-\dim \Lh$. 
\end{itemize}
\end{prop}  

The index of $M_{g,c}$ is given as follows.

\begin{thm}\label{5.2}
Let $\Lg=M_{g,c}$, with $g\ge 3$ and $c\ge 3$. Then we have
\begin{align*}
\chi(\Lg) & = \dim \Lg-2g \\[0.1cm]
          & = 2\al(\Lg)-\dim (\Lg).
\end{align*}
\end{thm}

\begin{proof}
Let $x_1,\ldots ,x_g$ be the generators of $M_{g,c}$, and $n_2=\dim M_{g,2}= \dim F_{g,2}$. Order the elements
$[x_i,x_j]$ for $1\le i<j\le g$ in some way and denote them by $x_{g+1},\ldots ,x_{n_2}$. Furthermore let
$n_3=\dim M_{g,3}= \dim F_{g,3}$ and denote the elements $[x_i,[x_j,x_k]]$ for $1\le j<k\le g,\, 1\le i\le k\le g$,
in some ordering, $x_{n_2+1},\ldots ,x_{n_3}$. Let $n=\dim M_{g,c}$ and choose a basis
$x_{n_3+1},\ldots ,x_n$ of $M_{g,c}^4$. Then $x_1,x_2,\ldots ,x_n$ is a basis of $M_{g,c}$. \\[0.2cm]
For the case $c=3$ we have shown in Theorem $\ref{3.4}$ that $\chi(M_{g,3})=n_3-2g$ for all $g\ge 3$.
Let $\Lh=F_{g,3}^2$. This subalgebra has a basis
\[
\{x_{g+1},\ldots ,x_{n_2},x_{n_2+1},\ldots ,x_{n_3}\}=\{h_1,h_2,\ldots ,h_{n_3-g}\}.
\]
Thus we have $\dim \Lh=n_3-g$, so that
\[
\chi(M_{g,3})=2\dim \Lh-n_3.
\]  
By Proposition $\ref{5.1}$ it follows that
\[
{\rm rank}_R([x_i,h_j])=n_3-\dim \Lh=g.
\]  
The matrix $([x_i,h_j])$ is a block matrix, with a principal block $([x_i,x_{g+j}])$ with $1\le i\le g$ and
$1\le j\le \binom{g}{2}$, and zero blocks otherwise. Note that $[x_i,x_{g+j}]\in \langle x_{n_2+1},\ldots ,x_{n_3}\rangle$.
It follows that also the principal block has rank $g$ over $R$, i.e.,
\[
{\rm rank}_R([x_i,x_{g+j}])=g.
\]
This will be used for the general case with $c\ge 4$. Let
\[
\Lh=M_{g,c}^2=\langle x_{g+1},\ldots , x_n\rangle = \langle h_1,\ldots ,h_{n-g} \rangle
\]
Then the matrix $([x_i,h_j])$ takes the block form
\[
\begin{pmatrix}
[x_i,x_{g+j}] & \vrule & [x_i,x_{n_2+k}] & \vrule & [x_i,x_{n_3+\ell}] \\
\hline \\[-0.48cm]
0 & \vrule & 0 & \vrule & 0 \\
\hline \\[-0.48cm]
0 & \vrule & 0 & \vrule & 0 \\
\hline \\[-0.48cm]
0 & \vrule & 0 & \vrule & 0
\end{pmatrix}
\]
where  $1\le i\le g$, $1\le j\le n_2-g$, $1\le k\le n_3-n_2$ and $1\le \ell\le n-n_3$.
The matrix in the top left corner is exactly the matrix we had obtained in the case $c=3$. Its rank equals $g$,
so that the rank of the above block matrix is also  equal to $g$. Since $g=\dim \Lg-\dim \Lh$, part $(1)$ in Proposition $\ref{5.1}$
is satisfied. Hence $(2)$ follows, i.e., we have
\begin{align*}
\chi(\Lg) & = 2\dim \Lh-\dim \Lg \\[0.1cm]
          & = 2(\dim \Lg-g)-\dim \Lg \\[0.1cm]
          & = \dim \Lg-2g.
\end{align*}  
By Theorem $14$ of \cite{OOM1}, $\al(\Lg)=\dim \Lh$, so that $\chi(\Lg)=2\al(\Lg)-\dim \Lg$ and $\al(\Lg)=\dim \Lg-g$.
\end{proof}  

Let us state the last result on $\al(\Lg)$ as a corollary.

\begin{cor}
Let $\Lg=M_{g,c}$ for $c\ge 3$ and $g\ge 3$. Then $\al(\Lg)=\dim \Lg-g$.
\end{cor}

It turns out that the statement of Theorem $\ref{5.2}$ is also true for $g=2$ and $c\ge 4$. However, we can no longer use
Theorem $\ref{3.4}$ in the proof, which is only valid for $g\ge 3$. Nevertheless we can adapt the proof to obtain the following result.

\begin{thm}\label{5.4}
Let $K$ be an algebraically closed field of characteristic zero, and $\Lg=M_{2,c}$ with $c\ge 4$. Then we have
\begin{align*}
\chi(\Lg) & = \dim \Lg-4 \\[0.1cm]
          & = 2\al(\Lg)-\dim (\Lg) \\[0.1cm]
          & = \frac{c^2-c-4}{2}
\end{align*}
\end{thm}

\begin{proof}
Let us consider the following basis for $\Lg$. We start with the generators $x_1,x_2$ and then take
\begin{align*}  
x_3 & = [x_2,x_1], \\
x_4 & = [[x_2,x_1],x_1]=[x_3,x_1], \\
x_5 & = [[x_2,x_1],x_2]=[x_3,x_2], \\
x_6 & = [[[x_2,x_1],x_1],x_1]=[x_4,x_1], \\
x_7 & = [[[x_2,x_1],x_1],x_2]=[x_4,x_2] \\
    & = [[[x_2,x_1],x_2],x_1]=[x_5,x_1], \\
x_8 & = [[[x_2,x_1],x_2],x_2]=[x_5,x_2], \\
\end{align*}   
and complete the basis by choosing $x_9,\ldots ,x_n\in \Lg^5$. See e.g.\ \cite{POR} for a complete basis for the free metabelian Lie algebra.
Now we can apply Proposition $\ref{5.1}$ with
\[
\Lh=\Lg^2=\langle x_3,x_4,\ldots ,x_n\rangle =\langle h_1,h_2,\ldots ,h_{n-2}\rangle.
\]
The matrix $([x_i,h_j])$ is of the following form:
\[
\begin{pmatrix}
-x_4 & -x_6 & -x_7 & \ast & \cdots & \ast \\
-x_5 & -x_7 & -x_8 & \ast & \cdots & \ast \\
0   & 0  &  0 & 0 & 0 & 0 \\
\vdots & \vdots & \vdots & \vdots & \ddots & \vdots \\
0   & 0  &  0 & 0 & 0 & 0 \\
\end{pmatrix}  
\]  
Clearly this matrix has rank $2$. Since $\dim \Lh=n-2=\dim \Lg-2$ we obtain
\[
{\rm rank}([x_i,h_j])=2=\dim \Lg-\dim \Lh.
\]
Hence  Proposition $\ref{5.1}$ yields
\[
\chi(\Lg)=2\dim \Lh-\dim \Lg.
\]
As in the last line of the proof of Proposition $5.2$ we have $\al(\Lg)=\dim \Lg-2$. It is well known (e.g.\ see the basis for the free metabelian Lie algebra as in \cite{POR}) that
\[
n=\dim \Lg=\frac{c^2-c+4}{2}.
\]
This finishes the proof.
\end{proof} 

\section{The index of filiform nilpotent Lie algebras}

Let $\Lg$ be a filiform Lie algebra of dimension $n\geq 3$ over a field $K$ of characteristic zero.
According to Vergne, see \cite{VER}, page $103$, $\Lg$ has an adapted basis, which is a  basis $\{e_1,e_2, \ldots ,e_n\}$
of $\Lg$ as a vector space satisfying:
\[\begin{array}{l}
{[e_1,e_i]}= e_{i+1} \mbox{ for } i= 2,3,\ldots , n-1\\[1mm]
{[e_1,e_n]}= 0 \\[1mm]
{[e_2,e_3]} \in \langle e_5,e_6, \ldots , e_n \rangle  
\end{array} 
\]
As shown in \cite{VER}, it follows also that there is an $\alpha\in K$ such that
\[\begin{array}{l}
{[e_i,e_j]}\in \langle e_{i+j}, e_{i,+j+1}, \ldots, e_n\rangle \mbox{ if  } i+j \leq n\\[1mm]
{[e_i,e_{n+1-i}]} = (-1)^i \alpha e_n \mbox{ for }i =2,3,\ldots n-1\mbox{ with }\alpha =0 \mbox{ if $n$ is odd.}  
\end{array} 
\]
For $i=1,2,\ldots, n$ we define the subspaces $\Lg_i= \langle e_i, e_{i+1}, \ldots, e_n\rangle$. 

\begin{lem}
Let $n\geq 4$. The spaces $\Lg_i$ are independent of the chosen adapted basis for $\Lg$. They are characteristic ideals of $\Lg$.
\end{lem}

\begin{proof}
It is clear that for $i\geq 3$ we have $\Lg_i=\Lg^{i-1}$, so these spaces are indeed independent of the chosen basis
and theses are  characteristic ideals of $\Lg$. For $i=2$, we consider the natural projection 
\[
p \colon \Lg \to \bar{\Lg}=\Lg/\Lg^4 .
\]
We use $\bar{e}_i=p(e_i)$ to denote the natural projection of the basis vectors. Then
\[
\bar{e}_5 = \bar{e}_6 = \cdots = \bar{e}_n =0
\]  
and $\{ \bar{e}_1, \bar{e}_2, \bar{e}_3, \bar{e}_4\}$ forms a basis  of $\bar{\Lg}$ and the only non trivial Lie brackets between these
basis vectors are 
\[ [ \bar{e}_1 , \bar{e}_2 ] = \bar{e}_3, \;\; [ \bar{e}_2 , \bar{e}_3 ] = \bar{e}_4. \]
Note that $\Lg_2=p^{-1} ( \langle \bar{e}_2, \bar{e}_3, \bar{e}_4\rangle )=p^{-1}(\bar{\Lg}_2)$.
But  $ \langle \bar{e}_2, \bar{e}_3, \bar{e}_4\rangle $ is a characteristic ideal of $\bar{\Lg}$ since
\[  \langle \bar{e}_2, \bar{e}_3, \bar{e}_4\rangle =C_{\bar{\Lg}}( \langle \bar{e}_3, \bar{e}_4\rangle ) =
C_{\bar{\Lg}}( \bar{\Lg}^2)=\{ x \in \bar\Lg \mid [x, \Lg^2 ] =0 \}
\]
and it is clear that this is a characteristic ideal of $\bar{\Lg}$. Moreover, this ideal is independent of the chosen basis, so
we can conclude that also $\Lg_2$ is independent of the chosen adapted basis and is a characteristic ideal of $\Lg$.
\end{proof}

\begin{defi}
Let $L_n$ be the standard graded filiform Lie algebra algebra with adapted basis $\{e_1,\ldots ,e_n\}$, with Lie brackets
$[e_1,e_i]=e_{i+1}$ for $2\le i\le n-1$, and all other brackets (not involving $e_1$) equal to zero. \\
For $n$ even, let $Q_n$ the graded filiform Lie algebra with adapted basis $\{e_1,\ldots ,e_n\}$, and nonzero Lie brackets defined by
\begin{align*}
[e_1,e_i] & = e_{i+1} \text{ for } 2\le i\le n-1, \\[0.1cm]
[e_i, e_{n+1-i}]  & =(-1)^i e_n \text{ for } 2\le i\le \frac{n}{2}.
\end{align*}  
\end{defi}  

The index of graded filiform Lie algebras has been determined in \cite{ADM}. 
In particular we have $\chi(L_n)=n-2$, and $\chi(Q_n)=2$. Lie algebras $\Lg$ with $\chi(\Lg)=0$ are {\em Frobenius Lie algebras}.
They have trivial center and hence cannot be nilpotent. So the minimal index of a filiform Lie algebra is equal to $1$.
The Heisenberg Lie algebra, which is $L_3$, has index $1$. We can give a characterization for filiform Lie algebras of index $1$
as follows. Recall that the dimension of a Lie algebra with odd index is necessarily odd.

\begin{thm} Let $\Lg$ be an odd dimensional filiform Lie algebra. Then we have $\chi(\Lg)=1$ if and only if
$[\Lg_i, \Lg_{n-i}] \neq 0$ for all $i=1,\ldots , n-1$.
\end{thm}

\begin{proof}
Let us fix an adapted basis $e_1,e_2, \ldots, e_n$ of a given odd dimensional filiform Lie algebra $\Lg$. Recall that
$[e_i,e_{n-i}] = \alpha_i e_n$ for some $\alpha_i$, with $\alpha_1=-\alpha_{n-1}=1$, and that $[e_i,e_j]=0$ as soon as
$i+j > n$ (since we are in the odd dimensional case). Now let $a=a_i e_i + a_{i+1} e_{i+1} + \cdots + a_n e_n \in \Lg_i$
and $b=b_{n-i} e_{n-i} + b_{n-i+1} e_{n-i+1} + \cdots + b_n e_n \in \Lg_{n-i}$. Then we have 
$[a,b]= a_i b_{n-i} \alpha_i e_n$, so $[\Lg_i , \Lg_{n-i}]=0$ if and only if $\alpha_i =0$. \\[0.2cm]
The index of $\Lg$ is $1$ if and only if we can find an element $\ell\in \Lg^\ast$ such that the subalgebra 
\[
\Lg(\ell) =\{ x \in \Lg \mid \ell ([y,x])=0 \mbox{ for all }y\in \Lg\}
\]
has dimension $1$. Let us consider an element $\ell \in \Lg^\ast$. Then $x =x_1 e_1 +x_2 e_2 + \cdots + x_n e_n$ belongs to $\Lg(\ell)$
if and only if $\ell([e_i,x])=0$ for all $i=1,2,\ldots,n-1$. This leads to a system of $n-1$ linear equations in the variables
$x_1,x_2, \ldots, x_{n-1}$ where the $i$-th equation is given by
\begin{eqnarray*}
0 = \ell([e_i,x]) & = & \ell ( x_1 [e_i,e_1] + x_2 [e_i,e_2] + \cdots + x_n[e_i, e_n]) \\
                           & = & \ell ( x_1 [e_i,e_1] + x_2 [e_i,e_2] + \cdots + x_{n-i} [e_i, e_{n-i}]) \\
                           & = & x_1 \ell ([e_i,e_1]) + x_2 \ell ([e_i,e_2]) + \cdots + x_{n-i-1} \ell([e_i, e_{n-i-1}]) + x_{n-i} \alpha_i \ell(e_n) 
\end{eqnarray*}
Note that indeed the variable $x_n$ does not appear in the system of equations, which reflects the fact that
$Z(\Lg) \subseteq \Lg(\ell)$ for all $\ell$.  So, $\Lg(\ell)$ is of dimension $1$ 
if and only if $Z(\Lg) = \Lg(\ell)$
which is equivalent to the fact that the system of linear equations in the variables $x_1,x_2, \ldots, x_{n-1}$ only
has the trivial solution $x_1=x_2=\cdots = x_{n-1} =0$.
The matrix corresponding to this system of equations is given by:
\[
\left(\begin{array}{cccccc}
\ell ([e_1,e_1]) & \ell([e_1,e_2] )& \cdots & \ell ( [e_1,e_{n-3}]) & \ell ([e_1, e_{n-2}] )& \alpha_1 \ell(e_n) \\
\ell([e_2,e_1]) & \ell([e_2,e_2]) &  \cdots & \ell(  [e_2,e_{n-3}]) & \alpha_2 \ell(e_n) & 0 \\
 \ell([e_3,e_1]) & \ell([e_3,e_2]) &  \cdots  & \alpha_3 \ell(e_n) & 0 & 0\\
\vdots & \vdots & \iddots & \vdots & \vdots & \vdots\\
\ell([e_{n-2}, e_1]) & \alpha_{n-2}\ell(e_n) & \cdots & 0 & 0 & 0 \\
\alpha_{n-1} \ell(e_n) & 0 & \cdots &  0 & 0 & 0 
      \end{array}\right)
\]
The system of equations only has the trivial solution if and only if the matrix has full rank, which is the case when $\ell(e_n)\neq 0$ and 
$\alpha_i\neq 0$, for all $i=1,2,\ldots , n-1$. We can always choose $\ell$ such that $\ell(e_n)\neq 0$, e.g.\, take $\ell = e_n^\ast$. 
The requirement  that $\alpha_i\neq 0$ for all $i$ is equivalent to the fact that  $[\Lg_i, \Lg_{n-i}]\neq 0$ for all $i$,
which is finishes the proof.
\end{proof} 

We can also obtain  a lower bound on the index of filiform Lie algebras as follows.

\begin{thm}\label{lower bound}
Let $\Lg$ be a filiform Lie algebra of dimension $n$. Suppose that $[\Lg_k, \Lg_k]=0$ for some $k\geq 2$. Then we have
$\chi(\Lg) \ge n -2 (k-1)$.
\end{thm}

\begin{proof}
As in the proof of the previous theorem we fix an adapted basis $\{e_1, e_2, \ldots , e_n\}$ for the filiform Lie algebra
$\Lg$ and consider an element $\ell \in \Lg^\ast$. Again an element
\[
x=x_1 e_1 +x_2 e_2 + \cdots + x_n e_n\in \Lg(\ell) 
\]
if and only if the variables $x_1, x_2, \ldots, x_{n-1}$ satisfy a system of $n-1$ linear equations, where the $i$-th equation
is given by 
\[
0 = \ell([e_i,x]) = x_1  \ell ([e_i,e_1]) + x_2 \ell ([e_i,e_2]) + \cdots + x_{n-1} \ell([e_i, e_{n-1}]).
\]
Let $A$ denote the $(n-1) \times (n-1)$ matrix associated to this system of linear equations, then
$A_{i,j}= \ell([e_i,e_j])$. The condition $[\Lg_k, \Lg_k]=0$ implies that $A_{i,j}=0$ as soon as $i\geq k $ and $j\geq k$.
So $A$ is of the form
\[
A= \left(\begin{array}{ccc|cccc}
A_{1,1} & \cdots & A_{1,k-1} & A_{1,k} & A_{1,k+1} & \cdots & A_{1,n-1} \\
\vdots &  \ddots & \vdots & \vdots & \vdots & \ddots & \vdots \\
A_{k-1,1}  & \cdots & A_{k-1,k-1} & A_{k-1,k} & A_{k-1,k+1} & \cdots & A_{k-1,n-1} \\ \hline
A_{k,1}  & \cdots & A_{k,k-1} & 0 & 0 & \cdots & 0\\
A_{k+1,1}  & \cdots & A_{k+1,k-1} & 0 & 0 & \cdots & 0 \\
\vdots  & \ddots & \vdots & \vdots & \vdots & \ddots & \vdots \\
A_{n-1,1} &  \cdots & A_{n-1,k-1} & 0 & 0 & \cdots & 0
\end{array} \right)
\]
From this it is clear that the rank of $A$ is at most $2(k-1)$ and so the dimension of $\Lg(\ell)$ is at least $n - 2(k-1)$. 
\end{proof}

\begin{rem}
We can give an alternative proof of this theorem by using the inequality
\[
\chi(\Lg)\ge 2\al(\Lg)-\dim \Lg
\]
from Proposition $\ref{2.3}$. For this, we may assume that $K$ is algebraically closed. By assumption $\Lg_k$ is an abelian
subalgebra of dimension $n-k+1$. This yields
\[
\chi(\Lg)\ge 2(n-k+1)-n=n-2(k-1). \qed
\]  
\end{rem}  

We now want to show that the bound on $\chi(\Lg)$ given in the theorem is sharp, for $2 k \leq n$. We will do this by constructing
a family of filiform Lie algebras $\Lg_{n,k}$ depending on two integer parameters $n$ and $k$ where $n\geq 3$ will be the dimension
of the filiform Lie algebra and $k$ is an odd number satisfying $3 \leq k \leq n$. \\[0.2cm]
The Lie algebras $\Lg_{n,k}$ will be constructed using semidirect products, so let us shortly recall the semidirect product
of two Lie algebras. Let $\La$ and $\Lb$ be two algebras where the respective brackets will be denoted by
$[\cdot, \cdot ]_{\La}$ and  $[\cdot, \cdot ]_{\Lb}$. Assume moreover that there is a Lie algebra morphisms
$\rho: \Lb \to \Der(\La)$, then we can define  the semidirect product Lie algebra $\La\rtimes_\rho \Lb$ which has
$\La \times \Lb$ as an underlying vector space and where the Lie bracket is given by 
\[
[(a_1,b_1), (a_2,b_2)] = ( [a_1,a_2]_\La + \rho(b_1) (a_2) -\rho(b_2)(a_1), [b_1,b_2]_\Lb).
\]
for all $a_1,a_2 \in \La$ and all $b_1,b_2\in \Lb$.
We now fix $n\geq 3$ and an odd integer $k$ with $3 \leq k \leq n$. We let $\La$ be the $(n-1)$-dimensional Lie algebra with
basis $\{\tilde{e}_2, \tilde{e}_3, \ldots , \tilde{e}_n\}$ and where the Lie brackets are given by 
\[ [\tilde{e}_i, \tilde{e}_j ]_\La= \left\{ \begin{array}{l}
0 \Leftrightarrow i+j \neq k\\[1mm]
(-1)^i \tilde{e}_n \Leftrightarrow i+j = k
\end{array}\right.\]
Note that these  brackets define an alternating bilinear product, since if $i+j=k$, then 
\[  [\tilde{e}_j, \tilde{e}_i]_\La = (-1)^j \tilde{e}_n = (-1) ^{k-i} \tilde{e}_n =- (-1)^i \tilde{e}_n = - [\tilde{e}_i, \tilde{e}_j]_\La,\]
where we used that $k$ is odd. As $[\tilde{e}_n,\tilde{e}_i]_\La=0$ for all $i=2,3,\cdots,n$ it is trivial to see that the Jacobi identity
is satisfied (every triple Lie bracket is 0), so $[\cdot, \cdot]_\La$ does define a Lie algebra structure on $\La$.
Note that for $k=3$, $\La$ is just the abelian Lie algebra, but this is also allowed. \\[0.2cm]
Now we define a linear map $D:\La \to \La$ with $D(\tilde{e}_i)= \tilde{e}_{i+1}$ for $i=2,3,\ldots, n-1$ and $D(\tilde{e}_n)=0$.
We claim that $D\in \Der(\La)$. For this we need to show that 
\[
D([\tilde{e}_i, \tilde{e}_j]_\La) = [D(\tilde{e}_i), \tilde{e}_j]_\La + [\tilde{e}_i, D(\tilde{e}_j)]_\La \mbox{ for }i,j \in \{2,3,\ldots, n\}.
\]
The left hand side of this equation is always zero, and the equation is trivially satisfied if $i=n$ or $j=n$, so we may assume
that $i<n$ and $j<n$. We have to show that 
\[
[D(\tilde{e}_i), \tilde{e}_j]_\La + [\tilde{e}_i, D(\tilde{e}_j)]_\La = [\tilde{e}_{i+1}, \tilde{e}_j]_\La + [\tilde{e}_{i}, \tilde{e}_{j+1}]_\La=0.
\]
If $i+j\neq k-1$, then $i+j+1\neq k$ and we have that $[\tilde{e}_{i+1}, \tilde{e}_j]_\La + [\tilde{e}_{i}, \tilde{e}_{j+1}]_\La=0+0=0$.
If on the other hand $i+j=k-1$ we have that 
$[\tilde{e}_{i+1}, \tilde{e}_j]_\La + [\tilde{e}_{i}, \tilde{e}_{j+1}]_\La = (-1)^{i+1} \tilde{e}_n + (-1)^i \tilde{e}_n =0$,
which shows that $D\in \Der(\La)$. \\[0.2cm]
We now consider the $1$-dimensional abelian Lie algebra with basis $\tilde{e}_1$ and we define the linear map
$\rho: \Lb \to \Der(\La)$ with $\rho(\tilde{e}_1) = D$. As $\Lb$ is $1$-dimensional, $\rho$ is automatically a morphisms
of Lie algebras and we obtain a semidirect product Lie algebra $ \Lg_{n,k} = \La \rtimes_\rho \Lb$. 
We will use $\{ e_1= (0, \tilde{e}_1),\; e_2 =(\tilde{e}_2, 0), \; e_3= (\tilde{e}_3, 0), \ldots, e_n= (\tilde{e}_n, 0) \}$
as a basis for $\Lg_{n,k}$. It is now easy to see that the non-zero Lie brackets between basis vectors are given by:
\[
[e_1, e_i]=e_{i+1} \; \; i=2,3,\ldots, n-1\mbox{ and } [e_i, e_{k-i}] =(-1)^i e_n \mbox{ for } 2 \leq i \leq k-2.
\]
It follows that $\Lg_{n,k}$ is indeed a filiform Lie algebra and $\{e_1,e_2, \ldots, e_n\}$ is an adapted basis for $\Lg_{n,k}$.
Note that for $k=3$, $\Lg_n$ is the standard graded filiform Lie algebra $L_n$. 

\begin{lem}\label{gl}
Let $n\geq 3$ and $k$ be an odd integer with $3 \leq k \leq n$ and let $\Lg = \Lg_{n,k}$. For $m=(k+1)/2$ it holds that $[\Lg_m , \Lg_m ]=0$. 
\end{lem}

\begin{proof}
The subalgebra $[\Lg_m , \Lg_m ]$ is spanned by elements $[e_i, e_j]$ with $i,j\geq m\geq 2$. In this case $i+j \geq k+1$ from which
it follows that   $i+j\neq k$ and so $[e_i,e_j]=0$.
\end{proof}

In the next proposition we compute the index of the Lie algebras $\Lg_{n,k}$ and thereby we will also show that the lower
bound obtained in Theorem~\ref{lower bound} is sharp. 

\begin{prop}\label{indexgnk}
Let $n\geq 3$ and let $k$ be an odd integer with $  3 \leq k \leq n$. The index of $\Lg_{n,k}$ is $n-k+1$.
\end{prop}

\begin{proof}
Let $m=(k+1)/2$, then Lemma~\ref{gl} and Theorem~\ref{lower bound} imply that the index of $\Lg_{n,k}$ is at least
$n - 2(m-1) = n-k+1$. So it suffices to find a linear form $\ell \in \Lg_{n,k}^\ast$ such that $\Lg_{n,k}(\ell)$ has
dimension $n-k+1$. Take $\ell = e_n^\ast$. As before, determining which $x=x_1 e_1 +x_2 e_2 + \cdots + x_n e_n$ belong
to $\Lg_{n,k}(\ell)$ leads to $n-1$ equations $\ell ([e_i, x])=0$, $i=1,2,\ldots n-1$, involving the variables
$x_1, x_2, \ldots ,x_{n-1}$. The first equation $i=1$ is $0= e_n^\ast ([e_1, x])=0$ which reduces to
\[
x_{n-1}=0.
\]
For $2\leq i \leq k-2$ we have  $0=e_n^\ast([e_i, x]) = e_n^\ast(-x_1 e_{i+1} + x_{k-i} (-1)^i e_{n}) = -(1)^i x_{k-i}$ so
these $k-2$ equations are 
\[ 
0 = x_{k-2}= x_{k-3}= \cdots = x_3= x_2.
\]
For $k-1 \leq i \leq n-2$ we get $0=e_n^\ast([e_i, x]) = e_n^\ast(-x_1 e_{i+1} ) =0$, so these are all the trivial equation $0=0$.
Finally, the last equation  $0=e_n^\ast[e_{n-1}, x] = e_n^\ast(-x_1 e_{n} )= -x_1$ (for $i=n-1$) leads to
\[
0 = x_1.
\]
As a conclusion, we have that  $x\in \Lg_{n,k}(\ell)$ if and only if
\[
x_1 =x_2 = \cdots = x_{k-3} = x_{k-2} = x_{n-1} =0.
\]
 So the dimension of $\Lg_{n,k}(\ell)$ is $n-(k-1)$ which was to be shown.
\end{proof}

\begin{cor}
Let $n\geq 3$ and let $i \in \{1,3,5,\ldots, n-2\}$ when $n$ is odd or $i \in \{2,4,6,\ldots , n-2\}$ when $n$ is even.
Then there exits a filiform Lie algebra $\Lg$ of dimension $n$ and with $\chi(\Lg)=i$.
\end{cor}

\begin{proof}
Take $k=n-i+1$. Then $k$ is odd and $3\leq k \leq n$ and therefore we can consider the Lie algebra $\Lg_{n,k}$. The index of
this Lie algebra is $n-k+1 = n - (n-i+1) +1 =i$.
\end{proof}

The above proof also shows the following result on the graded filiform Lie algebra $Q_n$, for $n$ even.

\begin{prop}
Let $\Lg$ be a filiform Lie algebra whose associated graded Lie algebra is isomorphic to $Q_n$. Then we have $\chi(\Lg)=2$.
\end{prop}

\begin{proof}
The fact that the associated graded Lie algebra of $\Lg$ is isomorphic to $Q_n$ means that $\Lg$ has an adapted basis in
which $[e_i, e_{n+1-i}] =(-1)^i e_n \mbox{ for } 2 \leq i \leq n-1$ holds.
By choosing $\ell=e_n^\ast$ and proceeding as before with the system of linear equations $0=e_n^\ast([e_i,x])$, one can see
that the dimension of $\Lg(\ell)$ is $2$, which implies that the index is also $2$.  
\end{proof}

\section*{Acknowledgments}
Dietrich Burde is supported by the FWF, the Austrian Science Foun\-da\-tion, with the grant DOI 10.55776/P33811. 
Karel Dekimpe is supported by Methusalem grant Meth/21/03, a long term structural funding of the Flemish Government.
For open access purposes, the authors have applied a CC BY public copyright license to any author-accepted manuscript version arising
from this submission.

\end{document}